\documentclass[reqno]{amsart}
\usepackage{amsmath,amsthm,amssymb,amscd}

\newtheorem{theorem}{Theorem}[section]
\newtheorem{lemma}[theorem]{Lemma}

\theoremstyle{definition}
\newtheorem{definition}[theorem]{Definition}

\theoremstyle{remark}
\newtheorem{remark}[theorem]{Remark}

\numberwithin{equation}{section}
\allowdisplaybreaks
\begin{document}

%-------------------------------------------------------------------------
% editorial commands: to be inserted by the editorial office
%
%\firstpage{1} \volume{228} \Copyrightyear{2004} \DOI{003-0001}
%\seriesextraline{This is the Concrete Title of this Book\br H.E. R and S.T.C. W, Eds.}
%\firstpage{1}

\title[On Mex-related partition functions of Andrews and Newman]
 {On Mex-related partition functions of Andrews and Newman}

%----------Author 1

\author{Rupam Barman}
\address{Department of Mathematics, Indian Institute of Technology Guwahati, Assam, India, PIN- 781039}
\email{rupam@iitg.ac.in}

\author{Ajit Singh}
\address{Department of Mathematics, Indian Institute of Technology Guwahati, Assam, India, PIN- 781039}
\email{ajit18@iitg.ac.in}

\date{August 30, 2020}

\subjclass{Primary 05A17, 11P83, 11F11}

\keywords{minimal excludant; mex function; partition; Eta-quotients; modular forms; arithmetic density}

\dedicatory{}
%%% ----------------------------------------------------------------------

\begin{abstract}
The minimal excludant, or ``mex'' function, on a set $S$ of positive integers is the least positive integer not in $S$. In a recent paper, Andrews and Newman extended the mex-function to integer partitions and found numerous surprising partition identities connected with these functions. Very recently, da Silva and Sellers present parity considerations of one of the families of functions Andrews and Newman studied, namely $p_{t,t}(n)$, and provide complete parity characterizations of $p_{1,1}(n)$ and $p_{3,3}(n)$. In this article, we study the parity of $p_{t,t}(n)$ when $t=2^{\alpha}, 3\cdot 2^{\alpha}$ for all $\alpha\geq 1$. We prove that $p_{2^{\alpha},2^{\alpha}}(n)$ and $p_{3\cdot2^{\alpha}, 3\cdot2^{\alpha}}(n)$ are almost always even for all $\alpha\geq 1$. Using a result of Ono and Taguchi on nilpotency of Hecke operators, we also find infinite families of congruences modulo $2$ satisfied by $p_{2^{\alpha},2^{\alpha}}(n)$ and $p_{3\cdot2^{\alpha}, 3\cdot2^{\alpha}}(n)$ for all $\alpha\geq 1$. 
\end{abstract}

%%% ----------------------------------------------------------------------
\maketitle
%%% ----------------------------------------------------------------------
%\tableofcontents
\section{Introduction and statement of results} 
For each set $S$ of positive integers the minimal excludant function (mex-function) is defined as follows: $$\text{mex}(S)=\text{min}(\mathbb{Z}_{>0}\setminus S).$$
Andrews and Newman \cite{Andrews-Newman} recently generalized this function to integer partitions. Given a partition $\lambda$ of $n$, they defined the mex-function $\text{mex}_{A, a}(\lambda)$ to be the smallest positive integer congruent to $a$ modulo $A$ that is not part of $\lambda$. They then defined $p_{A, a}(n)$ to be the number of partitions $\lambda$ of $n$ satisfying $$\text{mex}_{A, a}(\lambda)\equiv a \pmod{2A}.$$
For example, consider $n = 5$, $A = 2$, and $a = 2$. In the table below, we list the
seven partitions $\lambda$ of $5$ and the corresponding values of $\text{mex}_{2, 2}(\lambda)$ for each $\lambda$:\\
\begin{center}
\begin{tabular}{|c|c|}
	\hline
	Partition $\lambda$&$\text{mex}_{2, 2}(\lambda)$\\
	\hline 
	5&2\\
	$4+1$&2\\
	$3+2$&4\\
	$3+1+1$&2\\
	$2+2+1$&4\\
	$2+1+1+1$&4\\
	$1+1+1+1+1$&2\\
	\hline 
\end{tabular}
\end{center}
We see that four of the partitions of $5$ satisfy $\text{mex}_{2, 2}(\lambda)\equiv 2\pmod{4}$. Therefore, $p_{2, 2}(5)=4$. In \cite[Lemma 9]{Andrews-Newman}, Andrews and Newman proved that the generating function for $p_{t,t}(n)$ is given by
\begin{align}\label{gen-fun}
\sum_{n=0}^{\infty}p_{t,t}(n)q^n=\frac{1}{(q; q)_{\infty}}\sum_{n=0}^{\infty}(-1)^n q^{tn(n+1)/2},
\end{align}
where $\displaystyle (a; q)_{\infty}:= \prod_{j=0}^{\infty}(1-aq^j)$.
\par 
We note that $p_{1, 1}(n)$ and $p_{3,3}(n)$ are the sequences \cite[A064428]{integer-seq} and \cite[A260894]{integer-seq}, respectively. The arithmetic properties of the functions $p_{1, 1}(n)$ and $p_{3,3}(n)$ are studied in \cite{Andrews-Newman, sellers}. In order to state the results of Andrews and Newman on $p_{1, 1}(n)$ and $p_{3,3}(n)$, we now recall two partition statistics, the rank and the crank. The rank of a partition is the largest part minus the number of parts. The crank of a partition is the largest part of the partition if there are no ones as parts, and otherwise is the number of parts larger than the number of ones minus the numbers of ones. For more details on rank and crank, see for example \cite{dyson, garvan}. In \cite{Andrews-Newman}, Andrews and Newman proved that $p_{1, 1}(n)$ equals the number of partitions of $n$ with non-negative crank. They also proved that $p_{3,3}(n)$ equals the number of partitions of $n$ with rank $\geq -1$. Very recently, da Silva and Sellers \cite{sellers} provide complete parity characterizations of $p_{1, 1}(n)$ and $p_{3,3}(n)$. They prove that, for all $n\geq 1$,
\begin{align*}
p_{1, 1}(n)=\begin{cases}
1 \pmod{2}, & \mbox{if $n=k(3k\pm 1)$ for some $k$};\\
0\pmod{2}, & \mbox{otherwise}.\nonumber
\end{cases}
\end{align*}
Similarly, they prove that, for all $n\geq 1$,
\begin{align*}
p_{3, 3}(n)=\begin{cases}
1 \pmod{2}, & \mbox{if $3n+1$ is a square};\\
0\pmod{2}, & \mbox{otherwise}.\nonumber
\end{cases}
\end{align*}
\par In this article, we present parity results for $p_{2^{\alpha}, 2^{\alpha}}(n)$ and $p_{3\cdot2^{\alpha}, 3\cdot2^{\alpha}}(n)$ for all $\alpha\geq 1$. We find that the parities of $p_{2^{\alpha}, 2^{\alpha}}(n)$ and $p_{3\cdot2^{\alpha}, 3\cdot2^{\alpha}}(n)$ are different from that of $p_{1,1}(n)$ and $p_{3,3}(n)$. In the following two theorems, we prove that $p_{2^{\alpha}, 2^{\alpha}}(n)$ and $p_{3\cdot2^{\alpha}, 3\cdot2^{\alpha}}(n)$ are almost always even for all $\alpha\geq 1$. 
 %%%%%%%%%%%%%%%%%%%%%%%%%%%%
 \begin{theorem}\label{thm1}
 	For all $\alpha \geq 1$, the set $$\{n\in \mathbb{N}: p_{2^{\alpha}, 2^{\alpha}}(n)\equiv 0\pmod{2}\}$$ has arithmetic density $1$.
 \end{theorem}
\begin{theorem}\label{thm3}
	For all $\alpha \geq 1$, the set $$\{n\in \mathbb{N}: p_{3\cdot2^{\alpha}, 3\cdot2^{\alpha}}(n)\equiv 0\pmod{2}\}$$ has arithmetic density $1$.
\end{theorem}
 %%%%%%%%%%%%%
 Serre observed and Tate proved \cite{serre1, serre2, tate} that the action of Hecke algebras on spaces of modular forms of level 1 modulo 2 is locally nilpotent. Ono and Taguchi \cite{ono2005} showed that this phenomenon generalizes to higher levels. In this article, we find that the eta-quotients associated to the generating functions of $p_{2^{\alpha}, 2^{\alpha}}(n)$ and $p_{3\cdot2^{\alpha}, 3\cdot2^{\alpha}}(n)$ are modular forms whose levels land in Ono and Taguchi’s list. This allows us to use a result of Ono and Taguchi to prove the following congruences for $p_{2^{\alpha}, 2^{\alpha}}(n)$ and $p_{3\cdot 2^{\alpha}, 3\cdot 2^{\alpha}}(n)$ for all $\alpha\geq 1$.
 \begin{theorem}\label{thm2}
 	Let $\alpha$ be a positive integer. Then there is an integer $c_1 \geq 0$ such that for every $d_1 \geq 1$ and distinct primes $p_1, \ldots, p_{c_1+d_1}$ coprime to $6$, we have 
 	\begin{align*}
 	p_{2^{\alpha}, 2^{\alpha}}\left(\frac{p_1\cdots p_{c_1+d_1}\cdot n+1-3\cdot 2^{\alpha}}{24} \right)\equiv 0\pmod{2}
 	\end{align*}
 	whenever $n$ is coprime to $p_1, \ldots, p_{c_1+d_1}$.
 \end{theorem}
\begin{theorem}\label{thm4}
	Let $\alpha$ be a positive integer. Then there is an integer $c_2 \geq 0$ such that for every $d_2 \geq 1$ and distinct primes $p_1, \ldots, p_{c_2+d_2}$ coprime to $6$, we have 
	\begin{align*}
	p_{3\cdot2^{\alpha}, 3\cdot2^{\alpha}}\left(\frac{p_1\cdots p_{c_2+d_2}\cdot n+1-9\cdot 2^{\alpha}}{24} \right)\equiv 0\pmod{2}
	\end{align*}
	whenever $n$ is coprime to $p_1, \ldots, p_{c_2+d_2}$.
\end{theorem}
 %%%%%%%%%%%%%%%%%%%%%%%%%%%%%
\section{Preliminaries}
We recall some definitions and basic facts on modular forms. For more details, see for example \cite{ono2004, koblitz1993}. We first define the matrix groups 
\begin{align*}
\text{SL}_2(\mathbb{Z}) & :=\left\{\begin{bmatrix}
a  &  b \\
c  &  d      
\end{bmatrix}: a, b, c, d \in \mathbb{Z}, ad-bc=1
\right\},\\
\Gamma_{0}(N) & :=\left\{
\begin{bmatrix}
a  &  b \\
c  &  d      
\end{bmatrix} \in \text{SL}_2(\mathbb{Z}) : c\equiv 0\pmod N \right\},
\end{align*}
\begin{align*}
\Gamma_{1}(N) & :=\left\{
\begin{bmatrix}
a  &  b \\
c  &  d      
\end{bmatrix} \in \Gamma_0(N) : a\equiv d\equiv 1\pmod N \right\},
\end{align*}
and 
\begin{align*}\Gamma(N) & :=\left\{
\begin{bmatrix}
a  &  b \\
c  &  d      
\end{bmatrix} \in \text{SL}_2(\mathbb{Z}) : a\equiv d\equiv 1\pmod N, ~\text{and}~ b\equiv c\equiv 0\pmod N\right\},
\end{align*}
where $N$ is a positive integer. A subgroup $\Gamma$ of $\text{SL}_2(\mathbb{Z})$ is called a congruence subgroup if $\Gamma(N)\subseteq \Gamma$ for some $N$. The smallest $N$ such that $\Gamma(N)\subseteq \Gamma$
is called the level of $\Gamma$. For example, $\Gamma_0(N)$ and $\Gamma_1(N)$
are congruence subgroups of level $N$. 
\par Let $\mathbb{H}:=\{z\in \mathbb{C}: \text{Im}(z)>0\}$ be the upper half of the complex plane. The group $$\text{GL}_2^{+}(\mathbb{R})=\left\{\begin{bmatrix}
a  &  b \\
c  &  d      
\end{bmatrix}: a, b, c, d\in \mathbb{R}~\text{and}~ad-bc>0\right\}$$ acts on $\mathbb{H}$ by $\begin{bmatrix}
a  &  b \\
c  &  d      
\end{bmatrix} z=\displaystyle \frac{az+b}{cz+d}$.  
We identify $\infty$ with $\displaystyle\frac{1}{0}$ and define $\begin{bmatrix}
a  &  b \\
c  &  d      
\end{bmatrix} \displaystyle\frac{r}{s}=\displaystyle \frac{ar+bs}{cr+ds}$, where $\displaystyle\frac{r}{s}\in \mathbb{Q}\cup\{\infty\}$.
This gives an action of $\text{GL}_2^{+}(\mathbb{R})$ on the extended upper half-plane $\mathbb{H}^{\ast}=\mathbb{H}\cup\mathbb{Q}\cup\{\infty\}$. 
Suppose that $\Gamma$ is a congruence subgroup of $\text{SL}_2(\mathbb{Z})$. A cusp of $\Gamma$ is an equivalence class in $\mathbb{P}^1=\mathbb{Q}\cup\{\infty\}$ under the action of $\Gamma$.
\par The group $\text{GL}_2^{+}(\mathbb{R})$ also acts on functions $f: \mathbb{H}\rightarrow \mathbb{C}$. In particular, suppose that $\gamma=\begin{bmatrix}
a  &  b \\
c  &  d      
\end{bmatrix}\in \text{GL}_2^{+}(\mathbb{R})$. If $f(z)$ is a meromorphic function on $\mathbb{H}$ and $\ell$ is an integer, then define the slash operator $|_{\ell}$ by 
$$(f|_{\ell}\gamma)(z):=(\text{det}~{\gamma})^{\ell/2}(cz+d)^{-\ell}f(\gamma z).$$
\begin{definition}
	Let $\Gamma$ be a congruence subgroup of level $N$. A holomorphic function $f: \mathbb{H}\rightarrow \mathbb{C}$ is called a modular form with integer weight $\ell$ on $\Gamma$ if the following hold:
	\begin{enumerate}
		\item We have $$f\left(\displaystyle \frac{az+b}{cz+d}\right)=(cz+d)^{\ell}f(z)$$ for all $z\in \mathbb{H}$ and all $\begin{bmatrix}
		a  &  b \\
		c  &  d      
		\end{bmatrix} \in \Gamma$.
		\item If $\gamma\in \text{SL}_2(\mathbb{Z})$, then $(f|_{\ell}\gamma)(z)$ has a Fourier expansion of the form $$(f|_{\ell}\gamma)(z)=\displaystyle\sum_{n\geq 0}a_{\gamma}(n)q_N^n,$$
		where $q_N:=e^{2\pi iz/N}$.
	\end{enumerate}
\end{definition}
For a positive integer $\ell$, the complex vector space of modular forms of weight $\ell$ with respect to a congruence subgroup $\Gamma$ is denoted by $M_{\ell}(\Gamma)$.
\begin{definition}\cite[Definition 1.15]{ono2004}
	If $\chi$ is a Dirichlet character modulo $N$, then we say that a modular form $f\in M_{\ell}(\Gamma_1(N))$ has Nebentypus character $\chi$ if
	$$f\left( \frac{az+b}{cz+d}\right)=\chi(d)(cz+d)^{\ell}f(z)$$ for all $z\in \mathbb{H}$ and all $\begin{bmatrix}
	a  &  b \\
	c  &  d      
	\end{bmatrix} \in \Gamma_0(N)$. The space of such modular forms is denoted by $M_{\ell}(\Gamma_0(N), \chi)$. 
\end{definition}
In this paper, the relevant modular forms are those that arise from eta-quotients. Recall that the Dedekind's eta-function $\eta(z)$ is defined by
\begin{align*}
\eta(z):=q^{1/24}(q;q)_{\infty}=q^{1/24}\prod_{n=1}^{\infty}(1-q^n),
\end{align*}
where $q:=e^{2\pi iz}$ and $z\in \mathbb{H}$. A function $f(z)$ is called an eta-quotient if it is of the form
\begin{align*}
f(z)=\prod_{\delta\mid N}\eta(\delta z)^{r_\delta},
\end{align*}
where $N$ is a positive integer and $r_{\delta}$ is an integer. We now recall two theorems from \cite[p. 18]{ono2004} which will be used to prove our results.
\begin{theorem}\cite[Theorem 1.64]{ono2004}\label{thm_ono1} If $f(z)=\prod_{\delta\mid N}\eta(\delta z)^{r_\delta}$ 
is an eta-quotient such that $\ell=\frac{1}{2}\sum_{\delta\mid N}r_{\delta}\in \mathbb{Z}$, 
	$$\sum_{\delta\mid N} \delta r_{\delta}\equiv 0 \pmod{24}$$ and
	$$\sum_{\delta\mid N} \frac{N}{\delta}r_{\delta}\equiv 0 \pmod{24},$$
	then $f(z)$ satisfies $$f\left( \frac{az+b}{cz+d}\right)=\chi(d)(cz+d)^{\ell}f(z)$$
	for every  $\begin{bmatrix}
	a  &  b \\
	c  &  d      
	\end{bmatrix} \in \Gamma_0(N)$. Here the character $\chi$ is defined by $\chi(d):=\left(\frac{(-1)^{\ell} s}{d}\right)$, where $s:= \prod_{\delta\mid N}\delta^{r_{\delta}}$. 
\end{theorem}
Suppose that $f$ is an eta-quotient satisfying the conditions of Theorem \ref{thm_ono1} and that the associated weight $\ell$ is a positive integer. If $f(z)$ is holomorphic at all of the cusps of $\Gamma_0(N)$, then $f(z)\in M_{\ell}(\Gamma_0(N), \chi)$. The following theorem gives the necessary criterion for determining orders of an eta-quotient at cusps.
\begin{theorem}\cite[Theorem 1.65]{ono2004}\label{thm_ono2}
	Let $c, d$ and $N$ be positive integers with $d\mid N$ and $\gcd(c, d)=1$. If $f$ is an eta-quotient satisfying the conditions of Theorem \ref{thm_ono1} for $N$, then the 
	order of vanishing of $f(z)$ at the cusp $\frac{c}{d}$ 
	is $$\frac{N}{24}\sum_{\delta\mid N}\frac{\gcd(d,\delta)^2r_{\delta}}{\gcd(d,\frac{N}{d})d\delta}.$$
\end{theorem}
\par We next recall the definition of Hecke operators.
Let $m$ be a positive integer and $f(z) = \sum_{n=0}^{\infty} a(n)q^n \in M_{\ell}(\Gamma_0(N),\chi)$. Then the action of Hecke operator $T_m$ on $f(z)$ is defined by 
\begin{align*}
f(z)|T_m := \sum_{n=0}^{\infty} \left(\sum_{d\mid \gcd(n,m)}\chi(d)d^{\ell-1}a\left(\frac{nm}{d^2}\right)\right)q^n.
\end{align*}
In particular, if $m=p$ is prime, we have 
\begin{align*}
f(z)|T_p := \sum_{n=0}^{\infty} \left(a(pn)+\chi(p)p^{\ell-1}a\left(\frac{n}{p}\right)\right)q^n.
\end{align*}
We adopt the convention that $a(n)=0$ when $n$ is a non-negative integer.
\par We finally recall a density result of Serre which plays a crucial role in proving Theorem \ref{thm1} and Theorem \ref{thm3}. Using $\ell$-adic Galois representations attached to certain modular forms by Deligne, Serre \cite{serre3} proved the following remarkable theorem about the divisibility of Fourier coefficients of modular forms.
\begin{theorem}[Serre]\label{Serre}
Let $f(z)$ be a modular form of positive integer weight $k$ on some congruence subgroup of $SL_2(\mathbb{Z})$ with Fourier expansion $$f(z)=\sum_{n=0}^{\infty}a(n)q^n,$$
where $a(n)$ are algebraic integers in some number field. If $m$ is a positive integer, then there exists a constant $c>0$ such that there are  $O\left(\frac{X}{(\log X)^{c}}\right)$ integers $n \leq  x$ such that
$a(n)$ is not divisible by $m$.
\end{theorem}
%%%%%%%%%%%%%%%%%%%%%%%%%%%%%%%%%%%%%%
\section{Proof of Theorem \ref{thm1} and Theorem \ref{thm3}}
Let $\alpha$ be a positive integer. From \eqref{gen-fun}, we have 
\begin{align}\label{thm1.1}
\sum_{n=0}^{\infty} p_{2^{\alpha}, 2^{\alpha}}(n)q^{n}=\frac{1}{(q; q)_{\infty}} \sum_{n=0}^{\infty}(-1)^{n}q^{2^{\alpha}\cdot n(n+1)/2}.
\end{align}
For prime $p$ and positive integer $j$, it is easy to find that 
\begin{align}\label{binomial}
(1-q)^{p^{j}} \equiv\left(1-q^{p}\right)^{p^{j-1}} \pmod{p^{j}}.
\end{align}
Employing the Ramanujan's theta function
\begin{align*}
\psi(q):=\sum_{n=0}^{\infty} q^{n(n+1) / 2}=\frac{\left(q^{2} ; q^{2}\right)_{\infty}^{2}}{(q ; q)_{\infty}}
\end{align*}
into \eqref{thm1.1} and then using \eqref{binomial}, we find that
\begin{align}\label{thm1.2}
\sum_{n=0}^{\infty} p_{2^{\alpha},2^{\alpha}}(n) q^{n} &\equiv \frac{1}{(q; q)_{\infty}} \sum_{n=0}^{\infty} q^{2^{\alpha} n(n+1)/2}\pmod{2}\nonumber\\
&=\frac{1}{(q ; q)_{\infty}} \frac{\left(q^{2^{\alpha+1}} ; q^{2^{\alpha+1}}\right)_{\infty}^{2}}{\left(q^{2^{\alpha}} ; q^{2^{\alpha}}\right)_{\infty}}\nonumber \\
&\equiv \frac{1}{(q ; q)_{\infty}} \frac{\left(q^{2^{\alpha}} ; q^{2^{\alpha}}\right)_{\infty}^{4}}{\left(q^{2^{\alpha}} ; q^{2^{\alpha}}\right)_{\infty}}\pmod{2}\nonumber\\
&\equiv\frac{\left(q^{2^{\alpha}} ; q^{2^{\alpha}}\right)_{\infty}^{3}}{(q ; q)_{\infty}}\pmod{2}\nonumber\\
&\equiv q^\frac{{1-3\cdot 2^\alpha}}{24}\frac{\eta^3(2^{\alpha}z)}{\eta(z)}\pmod{2}.
\end{align}
Let 
\begin{align*}
G_{\alpha}(z) := \prod_{n=1}^{\infty} \frac{(1-q^{(3 \cdot 2^{\alpha+3 })n})^2}{(1-q^{(3\cdot 2^{\alpha+4 })n})} = \frac{\eta^2(3\cdot 2^{\alpha+3}z)}{\eta(3\cdot 2^{\alpha+4}z)}. 
\end{align*}
Then \eqref{binomial} yields
\begin{align}\label{g1}
G_{\alpha}^{2^{\alpha}}(z) = \frac{\eta^{2^{\alpha+1}}(3 \cdot2^{\alpha+3}z)}{\eta^{2^{\alpha}}(3\cdot 2^{\alpha+4}z)} \equiv 1 \pmod {2^{\alpha+1}}.
\end{align}
Define $H_{\alpha}(z)$ by
\begin{align*}
H_{\alpha}(z):= \left(\frac{\eta^3(3\cdot 2^{\alpha+3}z)}{\eta(24z)}\right)G_{\alpha}^{2^{\alpha}}(z)
=\frac{\eta^{3+2^{\alpha+1}}(3 \cdot 2^{\alpha+3}z)}{\eta(24z)\eta^{2^{\alpha}}(3\cdot 2^{\alpha+4}z)}.
\end{align*}
Due to \eqref{g1}, we have
\begin{align}\label{thm1.3}
H_{\alpha}(z) &\equiv \frac{\eta^3(3\cdot 2^{\alpha+3}z)}{\eta(24z)}\pmod{2^{\alpha+1}}\nonumber \\
&= q^{3\cdot 2^\alpha -1}\left(\frac{(q^{3\cdot 2^{\alpha+3}}; q^{3\cdot 2^{\alpha+3 }})^3_{\infty}}{(q^{24}; q^{24})_{\infty}}\right).
\end{align}
Combining \eqref{thm1.2} and \eqref{thm1.3}, we obtain 
\begin{align}\label{thm1.4}
H_{\alpha}(z) \equiv \sum_{n=0}^{\infty}p_{2^{\alpha}, 2^{\alpha}}(n)q^{24n+3\cdot 2^\alpha -1} \pmod {2}.
\end{align}
\begin{lemma}\label{lma1}
	Let $\alpha$ be a positive integer. Then $H_{\alpha}(z) \in M_{2^{\alpha -1}+1}\left(\Gamma_{0}(N), \chi_1 \right)$, where $N= 9\cdot 2^{\alpha +6}$ and the quadratic charecter $\chi_1$ is given by $\chi_1=(\frac{-2^{(\alpha+2)(2^{\alpha} +3)}3^{2^{\alpha} +2}}{\bullet})$.
\end{lemma}
\begin{proof}
	First we calculate the level of the eta-quotient $H_{\alpha}(z)$ by using Theorem \ref{thm_ono1}. The level of $H_{\alpha}(z)$ is equal to $3\cdot 2^{\alpha+ 4}\cdot m$, where $m$ is the smallest positive integer such that 
	\begin{align*}
	3\cdot 2^{\alpha +4}\cdot m\left[\frac{3+2^{\alpha+1}}{3\cdot 2^{\alpha+3}}-\frac{1}{24}-\frac{2^{\alpha}}{3\cdot 2^{\alpha+4}}\right]\equiv 0 \pmod{24}.
	\end{align*} 
	Equivalently,
	\begin{align*}
	m[6+2^{\alpha}]\equiv 0 \pmod{24}.
	\end{align*} 
	Therefore, $m=12$ and the level of $H_{\alpha}(z)$ is $9\cdot 2^{\alpha +6}$.	The cusps of $\Gamma_{0}(9\cdot2^{\alpha +6})$ are represented by fractions $\frac{c}{d}$ where $d\mid 9\cdot 2^{\alpha +6}$ and $\gcd(c, d)=1$. For example, see \cite[p. 12]{ono1996}. By Theorem \ref{thm_ono2}, we find that $H_{\alpha}(z)$ is holomorphic at a cusp $\frac{c}{d}$ if and only if
	\begin{align*}
	(2^{\alpha+1}+3)\frac{\gcd(d, 3\cdot 2^{\alpha+3})^2}{3\cdot 2^{\alpha+3}} 
	-\frac{\gcd(d, 24)^2}{24}-2^{\alpha}\frac{\gcd(d, 3\cdot 2^{\alpha+4})^2}{3\cdot 2^{\alpha+4}} \geq 0.
	\end{align*}
	Equivalently, if and only if
	\begin{align*}
	L:=(2^{\alpha+2}+6)\frac{\gcd(d, 3\cdot 2^{\alpha+3})^2}{\gcd(d, 3\cdot 2^{\alpha+4})^2} 
	-2^{\alpha+1}\frac{\gcd(d, 24)^2}{\gcd(d, 3\cdot 2^{\alpha+4})^2}-2^{\alpha} \geq 0.
	\end{align*}	
	In the following table, we find all the possible values of $L$.
	\begin{center}
		\begin{tabular}{|p{3.5cm}|p{1.88cm}|p{1.88cm}|p{3.0cm}|}
			%\hline
			%\multicolumn{4}{|c|}{Country List} \\
			\hline
			$d\mid 9\cdot 2^{\alpha +6}$ &$ \frac{\gcd(d, 3\cdot 2^{\alpha+3})^2}{\gcd(d, 3\cdot 2^{\alpha+4})^2}$ &$\frac{\gcd(d,24)^2}{\gcd(d, 3\cdot 2^{\alpha+4})^2}$& $L$\\	
			\hline	
			$1, 2, 3, 4, 6, 8,12,24,$ $ 9,$ $ 18, 36, 72$   & $1$   & $1$ & $6+ 2^{\alpha}$\\
			\hline
			$2^r$, $3\cdot 2^r,$ $9\cdot 2^r:$\newline 
			$ 4\leq r \leq \alpha+3$    & $1$ &   $1/2^{2r-6}$ &$6+ 3\cdot 2^{\alpha}- 2^{7+\alpha-2r}$\\
			\hline 
			$2^s$, $3\cdot 2^s,$ $9\cdot 2^s:$\newline 
			$\alpha+4\leq s \leq \alpha+6$& $1/4$  &$ 1/2^{2\alpha+2}$ &$2/3-1/2^{\alpha+1}$\\
			\hline
		\end{tabular}
	\end{center}
	Since $L\geq 0$ for all $d\mid 9\cdot2^{\alpha +6}$ and $\alpha \geq 1$, therefore $H_{\alpha}(z)$ is holomorphic at every cusp $\frac{c}{d}$. Using Theorem \ref{thm_ono1}, we find that the weight of $H_{\alpha}(z)$ is $\ell=2^{\alpha-1}+1$. Also, the associated character for $H_{\alpha}(z)$ is given by $\chi=(\frac{-2^{(\alpha+2)(2^{\alpha} +3)}3^{2^{\alpha} +2}}{\bullet})$.
	This completes the proof of the lemma. 
\end{proof}
%%%%%%%%%%%%%%%%%%%%%%%%%%%%%%%%%%%%%%%%%
\begin{proof}[Proof of Theorem \ref{thm1}] From Lemma \ref{lma1} we have $H_{\alpha}(z) \in M_{2^{\alpha-1}+1}\left(\Gamma_{0}(9 \cdot 2^{\alpha + 6}), \chi_1\right)$. Also, the Fourier coefficients of $H_{\alpha}(z)$ are all integers. Hence by Theorem \ref{Serre}, the Fourier coefficients of $H_{\alpha}(z)$ are almost always divisible by $m=2$. Due to \eqref{thm1.4}, the same holds for $p_{2^{\alpha},2^{\alpha}}(n)$. This completes the proof of the theorem.
\end{proof}
%%%%%%%%%%%%%%%%%%%%%
We next prove Theorem \ref{thm3}.
Let $\alpha$ be a positive integer. From \eqref{gen-fun}, we have 
\begin{align}\label{thm1.1.2}
\sum_{n=0}^{\infty} p_{3\cdot 2^{\alpha}, 3\cdot2^{\alpha}}(n)q^{n}=\frac{1}{(q; q)_{\infty}} \sum_{n=0}^{\infty}(-1)^{n}q^{3\cdot2^{\alpha}\cdot n(n+1)/2}.
\end{align}
Employing the Ramanujan's theta function
\begin{align*}
\psi(q):=\sum_{n=0}^{\infty} q^{n(n+1) / 2}=\frac{\left(q^{2} ; q^{2}\right)_{\infty}^{2}}{(q ; q)_{\infty}}
\end{align*}
into \eqref{thm1.1.2} and then using \eqref{binomial}, we find that
\begin{align}\label{thm1.2.2}
\sum_{n=0}^{\infty} p_{3\cdot2^{\alpha},3\cdot2^{\alpha}}(n) q^{n} &\equiv \frac{1}{(q; q)_{\infty}} \sum_{n=0}^{\infty} q^{3\cdot2^{\alpha} n(n+1)/2}\pmod{2}\nonumber\\
&=\frac{1}{(q ; q)_{\infty}} \frac{\left(q^{3\cdot2^{\alpha+1}} ; q^{3\cdot2^{\alpha+1}}\right)_{\infty}^{2}}{\left(q^{3\cdot2^{\alpha}} ; q^{3\cdot2^{\alpha}}\right)_{\infty}}\nonumber \\
&\equiv \frac{1}{(q ; q)_{\infty}} \frac{\left(q^{3\cdot2^{\alpha}} ; q^{3\cdot2^{\alpha}}\right)_{\infty}^{4}}{\left(q^{3\cdot2^{\alpha}} ; q^{3\cdot2^{\alpha}}\right)_{\infty}}\pmod{2}\nonumber\\
&\equiv\frac{\left(q^{3\cdot2^{\alpha}} ; q^{3\cdot2^{\alpha}}\right)_{\infty}^{3}}{(q ; q)_{\infty}}\pmod{2}\nonumber\\
&\equiv q^\frac{{1-9\cdot 2^\alpha}}{24}\frac{\eta^3(3\cdot2^{\alpha}z)}{\eta(z)}\pmod{2}.
\end{align}
Let 
\begin{align*}
R_{\alpha}(z) := \prod_{n=1}^{\infty} \frac{(1-q^{(9 \cdot 2^{\alpha+3 })n})^2}{(1-q^{(9\cdot 2^{\alpha+4 })n})} = \frac{\eta^2(9\cdot 2^{\alpha+3}z)}{\eta(9\cdot 2^{\alpha+4}z)}. 
\end{align*}
Then \eqref{binomial} yields
\begin{align}\label{r1}
R_{\alpha}^{2^{\alpha+1}}(z) = \frac{\eta^{2^{\alpha+2}}(9 \cdot2^{\alpha+3}z)}{\eta^{2^{\alpha+1}}(9\cdot 2^{\alpha+4}z)} \equiv 1 \pmod {2^{\alpha+2}}.
\end{align}
Define $S_{\alpha}(z)$ by
\begin{align*}
S_{\alpha}(z):= \left(\frac{\eta^3(9\cdot 2^{\alpha+3}z)}{\eta(24z)}\right)R_{\alpha}^{2^{\alpha+1}}(z)
=\frac{\eta^{3+2^{\alpha+2}}(9 \cdot 2^{\alpha+3}z)}{\eta(24z)\eta^{2^{\alpha+1}}(9\cdot 2^{\alpha+4}z)}.
\end{align*}
Due to \eqref{r1}, we have
\begin{align}\label{thm1.3.2}
S_{\alpha}(z) &\equiv \frac{\eta^3(9\cdot 2^{\alpha+3}z)}{\eta(24z)}\pmod{2^{\alpha+2}}\nonumber \\
&= q^{9\cdot 2^\alpha -1}\left(\frac{(q^{9\cdot 2^{\alpha+3}}; q^{9\cdot 2^{\alpha+3 }})^3_{\infty}}{(q^{24}; q^{24})_{\infty}}\right).
\end{align}
Combining \eqref{thm1.2.2} and \eqref{thm1.3.2}, we obtain 
\begin{align}\label{thm1.4.2}
S_{\alpha}(z) \equiv \sum_{n=0}^{\infty}p_{3\cdot2^{\alpha}, 3\cdot2^{\alpha}}(n)q^{24n+9\cdot 2^\alpha -1} \pmod {2}.
\end{align}
\begin{lemma}\label{lem2}
	Let $\alpha$ be a positive integer. Then $S_{\alpha}(z) \in M_{2^{\alpha }+1}\left(\Gamma_{0}(N), \chi_2 \right)$, where $N= 9\cdot 2^{\alpha +6}$ and the quadratic charecter $\chi_2$ is given by $\chi_2=(\frac{-2^{(\alpha+2)(2^{\alpha+1} +3)}3^{2^{\alpha+1} +2}}{\bullet})$.
\end{lemma}
\begin{proof}
The level of $S_{\alpha}(z)$ is equal to $9\cdot 2^{\alpha+ 4}\cdot m$, where $m$ is the smallest positive integer such that 
	\begin{align*}
	9\cdot 2^{\alpha +4}\cdot m\left[\frac{3+2^{\alpha+2}}{9\cdot 2^{\alpha+3}}-\frac{1}{24}-\frac{2^{\alpha+1}}{9\cdot 2^{\alpha+4}}\right]\equiv 0 \pmod{24}.
	\end{align*} 
	Equivalently,
	\begin{align*}
	6\cdot m\equiv 0 \pmod{24}.
	\end{align*} 
	Therefore, $m=4$ and the level of $S_{\alpha}(z)$ is $9\cdot 2^{\alpha +6}$. The cusps of $\Gamma_{0}(9\cdot2^{\alpha +6})$ are represented by fractions $\frac{c}{d}$ where $d\mid 9\cdot 2^{\alpha +6}$ and $\gcd(c, d)=1$. By Theorem \ref{thm_ono2}, we find that $S_{\alpha}(z)$ is holomorphic at a cusp $\frac{c}{d}$ if and only if
	\begin{align*}
	(2^{\alpha+2}+3)\frac{\gcd(d, 9\cdot 2^{\alpha+3})^2}{9\cdot 2^{\alpha+3}} 
	-\frac{\gcd(d, 24)^2}{24}-2^{\alpha+1}\frac{\gcd(d, 9\cdot 2^{\alpha+4})^2}{9\cdot 2^{\alpha+4}} \geq 0.
	\end{align*}
	Equivalently, if and only if
	\begin{align*}
	K:=(2^{\alpha+3}+6)\frac{\gcd(d, 9\cdot 2^{\alpha+3})^2}{\gcd(d, 9\cdot 2^{\alpha+4})^2} 
	-3\cdot2^{\alpha+1}\frac{\gcd(d, 24)^2}{\gcd(d, 9\cdot 2^{\alpha+4})^2}-2^{\alpha+1} \geq 0.
	\end{align*}	
As shown in the proof of Theorem \ref{thm1}, we verify that $K\geq 0$ for all $d\mid 9\cdot2^{\alpha +6}$ and $\alpha\geq 1$. Hence, $S_{\alpha}(z)$ is holomorphic at every cusp $\frac{c}{d}$. Now using Theorem \ref{thm_ono1}, we find that the weight of $S_{\alpha}(z)$ is $\ell=2^{\alpha}+1$, and the associated character for $S_{\alpha}(z)$ is given by $\chi_2=(\frac{-2^{(\alpha+2)(2^{\alpha+1} +3)}3^{2^{\alpha+1} +2}}{\bullet})$. 
\end{proof}
\begin{proof}[Proof of Theorem \ref{thm3}] Since $S_{\alpha}(z) \in M_{2^{\alpha}+1}\left(\Gamma_{0}(9 \cdot 2^{\alpha + 6}), \chi_2\right)$ and its Fourier coefficients are integers, so by Theorem \ref{Serre}, the Fourier coefficients of $S_{\alpha}(z)$ are almost always divisible by $m=2$. Due to \eqref{thm1.4.2}, the same holds for $p_{3\cdot2^{\alpha},3\cdot2^{\alpha}}(n)$. This completes the proof of the theorem.
\end{proof}
%%%%%%%%%%%%%%%%%%%%%
\section{Proof of Theorem \ref{thm2} and Theorem \ref{thm4}}
In this section we prove Theorem \ref{thm2} and Theorem \ref{thm4} using a result of Ono and Taguchi on nilpotency of Hecke operators. We use the following result which is implied by a much general result of Ono and Taguchi \cite[Theorem 1.3]{ono2005}. This result was also used by Aricheta (see for example \cite[Theorem 4.5]{arichet2017}).
\begin{theorem}\label{ono}     
	Let $n$ be a nonnegative integer and $k$ be a positive integer. Let $\chi$ be a quadratic Dirichlet character of conductor $9\cdot 2^{n}$. There is an integer $c \geq 0$ such that for every $f(z) \in M_{k}\left(\Gamma_{0}\left(9\cdot 2^{a}\right), \chi\right) \cap \mathbb{Z}[[q]]$ and every $t \geq 1$
	\begin{align*}
	f(z)\left|T_{p_1}\right| T_{p_2}|\cdots| T_{p_{c+t}} \equiv 0 \pmod{2^t}
	\end{align*}
	whenever the primes $p_{1}, \ldots, p_{c+t}$ are coprime to $6$.
\end{theorem}
\begin{remark} Theorem 1.3 of Ono and Taguchi is stated for the space of cusps forms; however, there is a remark right after the theorem which guarantees that we can use their result for modular forms. Ono and Taguchi remarked that one merely needs to verify that the conclusion holds for the subspace of Eisenstein series. This is easily done using well known formulas for the Fourier expansions of Eisenstein series which are given in terms of generalized divisor functions.
\end{remark}
\begin{remark}
It is a well-known fact that the Kronecher symbol $\left(\frac{a}{\bullet}\right)$ is a Dirichlet character if and only if $a\not\equiv 3\pmod{4}$ (see for example \cite{Allouche}).
\end{remark}
\begin{proof}[Proof of Theorem \ref{thm2}]
From \eqref{thm1.4}, we have
\begin{align*}
H_{\alpha}(z) \equiv \sum_{n=0}^{\infty}p_{2^{\alpha}, 2^{\alpha}}(n)q^{24n+3\cdot 2^\alpha -1} \pmod {2}.
\end{align*} 
This yields
	\begin{align}\label{thm2.3}
	H_{\alpha}(z):=\sum_{n=0}^{\infty} A_{\alpha}(n) q^{n}\equiv\sum_{n=0}^{\infty}p_{2^{\alpha},2^{\alpha}} \left(\frac{n}{24}+\frac{1-3\cdot 2^{\alpha}}{24}\right) q^{n}\pmod{2}.
	\end{align}
Note that $H_{\alpha}(z) \in M_{2^{\alpha-1}+1}\left(\Gamma_{0}(9 \cdot 2^{\alpha + 6}), \chi_1\right)$.  Using Theorem \ref{ono} we find that there is an integer $c_1\geq 0$ such that for any $d_1 \geq 1$,
	\begin{align*}
	H_{\alpha}(z)\left|T_{p_1}\right| T_{p_2}|\cdots| T_{p_{c_1+d_1}} \equiv 0 \pmod{2}
	\end{align*}
	whenever the primes $p_{1}, \ldots, p_{c_1+d_1}$ are coprime to $6$. It follows from the definition of Hecke operators that if $p_{1}, \ldots, p_{c_1+d_1}$ are distinct primes and if $n$ is coprime to $p_{1} \cdots p_{c_1+d_1}$ then
	\begin{align}\label{thm2.4}
	A_{\alpha}\left(p_{1} \cdots p_{c_1+d_1}\cdot n\right) \equiv 0 \pmod{2}.
	\end{align}
Combining \eqref{thm2.3} and \eqref{thm2.4} we complete the proof of the theorem.
\end{proof}
\begin{proof}[Proof of Theorem \ref{thm4}]
From \eqref{thm1.4.2}, we have
\begin{align*}
S_{\alpha}(z) \equiv \sum_{n=0}^{\infty}p_{3\cdot2^{\alpha}, 3\cdot2^{\alpha}}(n)q^{24n+9\cdot 2^\alpha -1} \pmod {2}.
\end{align*} 
This yields
\begin{align}\label{thm2.3.4}
S_{\alpha}(z):=\sum_{n=0}^{\infty} B_{\alpha}(n) q^{n}\equiv\sum_{n=0}^{\infty}p_{3\cdot2^{\alpha},3\cdot2^{\alpha}} \left(\frac{n}{24}+\frac{1-9\cdot 2^{\alpha}}{24}\right) q^{n}\pmod{2}.
\end{align}
We now proceed similarly as shown in the proof of Theorem \ref{thm2}. Applying Theorem \ref{ono} to $S_{\alpha}(z)$ we find that there is an integer $c_2\geq 0$ such that for any $d_2 \geq 1$ and distinct primes $p_{1}, \ldots, p_{c_2+d_2}$ coprime to 6, 
\begin{align}\label{thm2.4.4}
B_{\alpha}\left(p_{1} \cdots p_{c_2+d_2}\cdot n\right) \equiv 0 \pmod{2}
\end{align}
whenever $n$ is coprime to $p_{1}, \ldots,  p_{c_2+d_2}$.  Combining \eqref{thm2.3.4} and \eqref{thm2.4.4} we complete the proof of the theorem.
\end{proof}
\section{Acknowledgements}
We are extremely grateful to Ken Ono for previewing a preliminary version of this paper and for his helpful comments. We are indebted to Victor Manuel Aricheta for many fruitful discussions while preparing this article.
%%%%%%%%%%%%%%%%%%%%%%%%%%%%%%%%%%%%%%%%%

%\bibliographystyle{acm}
%\bibliographystyle{sej}
%\bibliographystyle{plain}
%\bibliography{ref.cray.bib}
\end{document}